\documentclass{amsart}
\usepackage{amssymb,latexsym} 
\usepackage{color}

 \newtheorem{thm}{Theorem}[section]
 \newtheorem{cor}[thm]{Corollary}
 \newtheorem{lem}[thm]{Lemma}
 \newtheorem{prop}[thm]{Proposition}
 \theoremstyle{definition}
 \newtheorem{defn}[thm]{Definition}
 \newtheorem{rem}[thm]{Remark}
 \theoremstyle{definition}
  \newtheorem{conj}[thm]{Conjecture}

 \DeclareMathOperator{\pfaff}{pfaff}

 \newcommand{\CC}{\mathbb{C}}
 
  \newcommand{\PP}{\mathbb{P}}
    \newcommand{\VV}{\mathcal{V}}

\begin{document}

\title[Expressing Forms as a Sum of Pfaffians]
{Expressing Forms as a Sum of Pfaffians}

\author[L. Chiantini ]{Luca Chiantini}
\footnote{Let me express my gratitude for this opportunity to leave a
tribute in honor of Emilia. We met at the beginning of our 
careers and since then I feel that we share the 
same global view of Mathematics and its applications.}
\address[L. Chiantini]{Universit\`a degli Studi di Siena, 
Dipartimento di Ingegneria dell'Informazione e Scienze Matematiche, 
Via Roma 56 (S.Niccolo'), 53100 Siena, Italy.}
\email{luca.chiantini@unisi.it}


\begin{abstract} 
{Let $A= (a_{ij})$ be a symmetric non-negative integer $2k\times 2k$ matrix.
$A$ is homogeneous if $a_{ij} + a_{kl}=a_{il} + a_{kj}$ for any choice 
of the four indexes.  Let $A$ be a homogeneous matrix and let $F$ be a 
\emph{general} form in $\CC[x_1, \dots x_n]$ with $2\deg(F) = {\rm trace}(A)$. 
We look for the least integer, $s(A)$, so that $F = \pfaff(M_1) + \cdots + \pfaff(M_{s(A)})$, where the $M_i = (F^i_{lm})$ are $2k\times 2k$ 
skew-symmetric matrices of forms with degree matrix $A$.
We consider this problem for $n= 4$ and we prove that $s(A) \leq k$ 
for all $A$.}
\end{abstract}


\date{\today}
\maketitle

\section*{Introduction}

Let $F \in \CC[x_1, \ldots , x_n]$ be a general form and 
$A = (a_{ij})$  a $2k\times 2k$ integer homogeneous symmetric 
matrix, whose  trace (${\rm{tr}}(A)$ in the sequel) is 
equal to twice the degree of $F$ ($\deg F$). In this 
paper we study representations of $F$ as a sum of 
pfaffians of skew-symmetric matrices of type 
$M = (F_{ij})$ where $\deg F_{ij} = a_{ij}$.  

In case the number of variables is two then 
forms $F$ in $\CC[x_1,x_2]$  
decompose as a product of linear forms.  
It follows that if $A$ is a matrix as above, with no 
negative entries and with ${\rm tr}(A) = 2\deg(F)$, 
then $F$ is the pfaffian of a subdiagonal matrix 
whose degree matrix is $A$ (i.e. a matrix of type
$$\begin{pmatrix}
0    & p_1 & 0    & 0   & 0 & \dots & 0 \\
-p_1 & 0   & 0    & 0   & 0 & \dots & 0 \\
0    & 0   & 0    & p_2 & 0 & \dots & 0 \\
0    & 0   & -p_2 & 0   & 0 & \dots & 0 \\
\dots & \dots & & & & \dots & \dots \\
0    & 0   & 0    & 0   & 0 & 0    & p_k \\
0    & 0   & 0    & 0   & 0 & -p_k & 0
\end{pmatrix}
$$
with each $p_i$ equal to a suitable product of linear forms). 

For 3 variables, the problem was considered by Beauville 
who observed, in section $5$ of \cite{Beau},
that a general form of degree $d$ is the pfaffian of a $2d\times 2d$ 
skew-symmetric matrix of linear forms.  
Indeed Beauville's argument applies to any symmetric integer 
homogeneous matrix with non-negative entries.
We give below a proof of the result, in a more geometric setting 
(see section \ref{3var}).

When the number of variables grows, then  a similar property fails 
as soon as $k$ becomes big.
In Proposition 7.6 of \cite{Beau} Beauville noticed that one cannot expect 
that a general form of degree $\geq 16$ in four variables
is the pfaffian of a matrix of  linear forms, 
just by a count of parameters. 
We refer to \cite{Faen} for a similar result for 
matrices of quadratic forms, and to \cite{CF2}
for an extension to other constant or almost constant 
matrices. In any setting, except for particular
numerical cases (which become suddenly unbalanced 
when the size of the matrix grows),
one expects that a general form is \emph{not} the pfaffian 
of a skew-symmetric matrix
of forms with fixed degrees. Indeed, even in the case of 
$4\times 4$ matrices and $4$ variables,
we do not know a complete description of  matrices $A=(a_{ij})$, 
with trace $2d$, such that the general form of degree  $d$
is the pfaffian of a skew-symmetric matrix of 
forms $(F_{ij})$ with $\deg(F_{ij})=a_{ij}$.
The problem seems rather laborious, and we 
refer to \cite{CCG2} for a discussion.

The problem is indeed related to the existence of 
indecomposable rank $2$ bundles $E$ without intermediate 
cohomology (aCM bundles) on the hypersurface 
defined by $F=0$ (which we will indicate, by abuse, with 
the same letter $F$). In turn, this is
equivalent to the existence of some arithmetical Gorenstein
subscheme of codimension $2$ in $F$ (thus codimension $3$ 
in the projective space), via the algebraic characterization of 
codimension $3$  Gorenstein ideals, given in \cite{BE}.
For instance, $F$ is a pfaffian of a $4\times 4$ 
skew-symmetric matrix of forms if and only if there exists
a subscheme of $F$ which 
is complete intersection of $3$ forms, whose degrees are 
related with the degrees of the entries of the matrix. 
This is the point of view under which the
problem is attached in \cite{CCG2}, and see 
also \cite{MKRR} for a similar discussion. 

In the present note, we make one step further. 
Since in most cases one cannot hope to express a general 
form as the pfaffian of a skew-symmetric matrix 
of forms with pre-assigned degree matrix $A$,
then we ask for the minimum $s(A)$ such that 
a general form is a {\it sum of $s(A)$} pfaffians
of skew-symmetric matrices, with degree matrix $A$.
 
We consider the case of forms in four variables 
and show that the complete answer $s(A)\leq 2$ follows soon 
for  $4\times 4$ matrices $A$, while for 
$2k\times 2k$ matrices with $k>2$, we provide a bound 
for the number $s(A)$, i.e. $s(A)\leq 2k$. The (weak)
sharpness of this bound is discussed in the last section.
As showed in \cite{Beau} and \cite{CF2}, at least 
for small values of the entries
of the integer matrix $A$ (e.g. for matrices of linear forms), the 
number of pfaffians needed to write a general form
can be smaller than our bound. The problem of finding a {\it sharp}
bound for the number $s(A)$ is open.

The procedure is a mixture of algebraic and geometric 
arguments, involving  computations
of the dimension of secant varieties and 
Terracini's Lemma, as well as the description 
of  tangent spaces to the varieties of forms 
that can be expressed as pfaffians, 
given in \cite{KL}, \cite{Kleppe} or \cite{Adler}.
  
We mention that, of course, one could ask a similar 
question for the {\it determinant} 
of a general matrix of forms. In other words, 
fixing a homogeneous integer matrix $A$, 
one could ask for  the minimum $s'(A)$ such that the general
form of degree $d=tr(A)$ is the sum of the determinants 
of $s'(A)$ matrices of forms, with degree
matrix $A$. This is indeed the target of a 
series of papers \cite{CM}, \cite{CG1}, \cite{CG2},
where it is proved that, in $n\geq 3$ variables, 
 $s'(A)\leq k^{n-3}$ for a $k\times k$ matrix $A$.  
 
Let us end by noticing that the problem addressed in this note, 
of clear algebraic and geometric flavor, turns out to also 
have a connection with some applications in control theory. 
Indeed, if the algebraic boundary of a region $\Theta$
in the plane or in space is described by the pfaffian
of a matrix of linear forms, then the study of systems
of matrix inequalities, whose domain is
$\Theta$, can be considerably simplified. We refer to the papers
\cite{V89} and \cite{HL12}, for an account of this theory. 
We believe that expressing $\Theta$ as a sum of 
determinants or pfaffians
can have some application for similar problems.

 \section{The geometric construction}\label{geom}

We work in the ring $R = \CC[x_0,\dots ,x_n]$, i.e. 
the polynomial ring in $n+1$ variables
with coefficients in the complex field.  
By  $R_d$ we indicate the vector space
of homogeneous forms of degree $d$ in $R$.

For any degree $d$, the space $R_d$  
has an associated projective space $\PP^N$ with 
$$N:=N(d)= \binom{n+d}n -1.$$ 

For any choice of integers $a_{ij}$, $1\leq i,j\leq 2k$,
consider the numerical $2k\times 2k$ matrix $A=(a_{ij})$.

We will say that a $2k\times 2k$ matrix $M=(F_{ij})$, 
whose entries are homogeneous
forms in $R$, has {\it degree matrix $A$} if for all
$i,j$ we have $\deg(F_{ij})=a_{ij}$. 
In this case, we will also write that $A=\partial M$.

Notice that when for some $i,j$ we have 
$F_{ij}=0$, then there are several possible degree matrices for $M$, 
since the degree of the zero polynomial is indeterminate. 

We will focus on the case where $A$ is symmetric 
and $M$ is skew-symmetric.

The set of all skew-symmetric matrices of forms, whose degree matrix
is a fixed $A$, defines a vector space whose dimension 
is  $\sum_{i<j} \dim(R_{a_{ij}})$. From the geometrical point of
view, however, we will consider this set as the {\it product}
of projective spaces
$$ \VV(A) = \PP^{r_{12}}\times\cdots\times\PP^{r_{2k-1\, 2k}}$$
where $r_{ij}=-1+\dim(R_{a_{ij}})$. 
\smallskip

We say that the numerical matrix $A$ is {\it homogeneous} when,
for any choice of the indexes $i,j,l,m$, we have
$$ a_{ij}+a_{lm} = a_{im}+a_{lj}.$$

All submatrices of a homogeneous matrix are homogeneous.

If a skew-symmetric $2k\times 2k$ matrix of forms $M$ has a 
homogeneous degree matrix, then
the pfaffian of $M$ is a homogeneous form. 
The degree of the pfaffian is one half of the sum of the numbers 
on the main diagonal of $A=\partial M$, i.e. $tr(A)/2$. It is indeed 
immediate to see that when $A$ is symmetric and homogeneous 
of even size, then the trace $tr(A)$ is even. 
\smallskip

Let us recall  a geometric interpretation 
of the problem, based on the study of secant varieties, 
which uses the classical Terracini's  Lemma. This   
a standard construction was already used in \cite{CCG2}.   

In the {\it projective} space $\PP^N$, which parametrizes
all forms of degree $d$, we have the subset $U$ of all the forms
which are the pfaffian of a skew-symmetric matrix of forms whose degree
matrix is a given $A$.  This set is a quasi-projective variety,
since it corresponds to the image of the (rational) map $\VV(A)\to\PP^N$,
which sends every matrix to its pfaffian (it is undefined
when the pfaffian is the zero polynomial).
We will denote by $V$ the closure of $U$. 
It is clear that $V$ is irreducible, by construction.

Our main question can be rephrased by asking: what is the
minimal $s$ such that a general point of $\PP^N$ is spanned by
$s$ points of $V$? In classical Algebraic Geometry,
(the closure of) the set of points spanned by $s$ points of
$V$ is called the {\it $s$-th secant variety} $\sigma_s(V)$ of $V$.
Thus, we look for the minimal $s$ such that $\sigma_s(V)=\PP^N$.
Of course, this is equivalent to ask that the 
dimension of $\sigma_s(V)$ is $N$.

Usually, when dealing with similar problems on secant varieties, 
one can hope to compute the dimension of $\sigma_s(V)$ 
as the dimension of a general tangent space to $\sigma_s(V)$. 
Indeed one can invoke the celebrated Terracini's Lemma:

\begin{lem} \emph{(Terracini)} At a general point $F\in\sigma_s(V)$, 
expressed as a sum $F=F_1+\cdots +F_s$, $F_i\in V$ for all $i$, 
the tangent space to $\sigma_s(V)$
equals the span of the tangent spaces to $V$ at $F_1,\dots,F_s$.
\end{lem}

Thus, for our purposes, it is crucial to obtain a characterization
of the tangent space to $V$ at a general point $F$. 
This has been obtained  (see e.g. \cite{KL})
via the submaximal pfaffians of matrices.

\begin{defn}\label{sub}
Let $M$ be a skew-symmetric matrix.

If the size of $M$ is even, 
 we denote as \emph{ submaximal pfaffians} 
the pfaffians  of the (skew-symmetric) submatrices 
of $M$ obtained by erasing two rows and the two 
columns with the same indexes.

If the size of $M$ is odd,  
we denote as \emph{ submaximal pfaffians} 
the pfaffians  of the (skew-symmetric) submatrices 
of $M$ obtained by erasing one row and one 
column with the same index.
\end{defn}

Then we have the following.

\begin{prop}\label{tg} 
With the previous notation, let $F$ be a general element in $V$, 
$F=\pfaff( M)$, where $M=(F_{ij})$ is a $2k\times 2k$ 
skew-symmetric matrix of forms, whose degree matrix is $A$.

Then the tangent space to $V$ at $F$ coincides with
the subspace of $R_d/\langle F \rangle$, generated by the classes
of the forms of degree $d$ in the ideal $J=\langle F, M_{ij}\rangle$,
where the $M_{ij}$'s are the submaximal pfaffians of the matrix $M$. 
\end{prop}
\begin{proof} See  \cite{Adler} or section 2 of \cite{KL} or \cite{KL}.
It can be obtained also by a direct computation over 
the ring of dual numbers.
\end{proof}

For instance, when the degree matrix $A$ of $M$ has 
all entries equal to $a$, then 
$J$ is generated by $\binom {2k} 2$ forms of degree $a(k-1)$.

It follows immediately from the previous propositions
and Terracini's lemma, that:

\begin{rem}\label{reduc}  
We have the following equivalences:

- a general form of degree $d$ is the sum of $s$  
pfaffians of $2k\times 2k$ matrices, all having degree matrix $A$ 

if and only if

- the span of $s$ general tangent spaces to the variety $V$ of
pfaffians is the whole space $\PP^N$

if and only if
  
- for a general choice of $s$ matrices of forms $M_1,\dots M_s$, of
size $2k\times 2k$, with $\partial M_i=A$ for all $i$, the ideal 
generated  by  the submaximal pfaffians of {\it all} 
the $M_i$'s coincides, in degree $d$, with
the whole space $R_d$.
\end{rem}

Thus, what we are looking for is the minimal $s$ such that, 
for general skew-symmetric matrices
$G_1,\dots, G_s$ with degree matrix $A$, the ideal $I$ generated
by  their submaximal pfaffians coincides with 
the whole polynomial ring  in degree $d=tr(A)/2$.

\begin{rem}\label{order}
If $M$ is a $2k\times 2k$ skew-symmetric matrix of 
forms with homogeneous degree matrix $A$, then 
the pfaffian of $M$ is essentially invariant if we permute 
rows and  columns of $M$ with the same indexes.
Consequently, we can arrange $A=(a_{ij})$ so that 
$$ a_{11}\geq a_{21}\geq \cdots \geq a_{2k\, 1}.$$
We will say that $A$ is {\it ordered} if it satisfies 
the previous inequalities.

Notice that $A$ is symmetric, thus if $A$ is ordered then
 $ a_{11}\geq a_{12}\geq \cdots \geq a_{1\, 2k}.$ 

Since $A$ is homogeneous, when $A$ is ordered $a_{ij}\geq a_{i'j}$ 
for some $j$ implies that $a_{ij'}\geq a_{i'j'}$ for any $j'$.
It follows that  a homogeneous symmetric ordered matrix $A$ has
$a_{11}$ as its maximal entry and $a_{2k\, 2k}$ as its minimal one. 
Moreover columns are non-increasing going downward, 
while rows are non-increasing going rightward.

Notice also that when $A$ is symmetric and homogeneous, then
the entries of the diagonal of $A$ are either all odd or all even. 
Indeed for any $i,j$ one has 
$$a_{ii}+a_{jj}=a_{ij}+a_{ji} = 2a_{ij}.$$
\end{rem}

\section{The case of ternary forms}\label{3var}
 
As mentioned in the introduction, the fact that any form of degree $d$ in
$2$ variables is the pfaffian of a skew-symmetric matrix with prescribed degree
matrix $A$ is trivial. Thus the first relevant case concerns forms in three
variables. 

For three variables, the construction of pfaffian representations of forms
via the existence of aCM rank $2$ bundles, given by Beauville in section 5
of \cite{Beau}, proves that the following holds:

\begin{thm}\label{tre} Let $A=(a_{ij})$ be a non-negative 
symmetric homogeneous integer matrix of even size,
with trace $2d$. Then a general homogeneous form of 
degree $d$ in three variables 
is the pfaffian of a skew-symmetric matrix 
of forms $G$ with $\partial G=A$.
\end{thm} 
 
Indeed, Beauville states the theorem only for matrices 
of linear forms. For completeness, we show an 
inductive method which, starting with Beauville's claim,
proves the statement for any non-negative matrix $A$.

We have the chance, in this way, to introduce our inductive method
for the study of pfaffian representations of forms in more variables.

Let us start with  a $(2k-1)\times (2k-1)$ integer 
matrix $A'=(a'_{ij})$, which is moreover symmetric, 
non-negative, {\it ordered} and homogeneous.
  Notice that, by the homogeneity assumption, 
  the trace of $A'$ is also equal to 
 $$ tr(A')= a_{12}+a_{23}+\dots +a_{2k-2\, 2k-1}+a_{2k\, 1}.$$

Let $G'$ be a skew-symmetric matrix of {\it general} 
forms, with degree matrix $A'$.
The {\it submaximal} pfaffians of $G'$ (see Definition
\ref{sub}) determine an ideal $I(G')$ 
whose zero-locus is  an arithmetically Gorenstein 
subscheme of codimension $3$, by the celebrated structure 
theorem of Buchsbaum and Eisenbud
(\cite{BE}). Moreover, we have a resolution of $I(G')$ of type
$$ 0\to R'(-m)\to \oplus R'(-r_j)\to \oplus R'(-s_i)\to I(G')\to 0$$
where $R'$ is the polynomial ring 
$R':=\CC[x,y,z]$, and $m$ is the trace of $A'$.
Since we are working in dimension $2$, the ideal $I(G')$ defines 
the empty set in $\PP^2$, 
and the resolution shows that $I(G')$ coincides with the 
whole polynomial ring in all degrees $d\geq m-2$.
\medskip

\begin{lem}\label{WL2}
Let $G$ be a general skew-symmetric matrix of forms
in three variables, of odd size $(2k-1)\times (2k-1)$,  
whose degree-matrix $A$ is non-negative  and homogeneous. 
Call $I$ the ideal generated by the submaximal pfaffians of $G$. 
Then the multiplication map by a general linear form $L$
defines a surjection $(R'/I)_{d-1}\to (R'/I)_d$ for all
$d\geq tr(A)/2 -1$.
\end{lem}
\begin{proof} Since $G$ is general, by \cite{H} $R'/I$ is artinian and 
arithmetically Gorenstein and enjoys the weak Lefschetz property.
The conclusion follows since the socle degree of $R'/I$
is at most $tr(A)-2$.
\end{proof}

\par\noindent
{\it Proof of Theorem \ref{tre}.} 
We may assume that $A$ is ordered and we will make 
induction on the trace of $A$. As explained in Remark \ref{order}, 
the entries of the diagonal of  $A$ are either all even or all odd.
If the entries  are even, we use as basis for the induction 
the null matrix, for which the statement is trivial. If the 
entries are odd,  we use the matrix with all the entries equal to $1$, 
for which the statement holds by \cite{Beau} Proposition 5.1.

For the inductive step, let $B$ be the matrix obtained from $A$ 
by subtracting $1$ to  the first row and the first column
(hence subtracting $2$ from $a_{11}$). We have
$tr(B)=tr(A)-2$ and by induction the theorem holds for $B$.
Thus if $H$ is a general matrix of forms with degree matrix
$B$, then by Remark \ref{reduc} the submaximal pfaffians
of $H$ generate an ideal $I(H)$ which coincides with $R'$ in degree
$\geq tr(A)/2-1$. 

Consider the symmetric matrix $G'$ obtained from $H$ by erasing 
the first row and the  first column. Then $G'$ is a general 
skew-symmetric $(2k-1)\times (2k-1)$ matrix of forms,
whose degree matrix $A'$ corresponds to $A$ minus the
first row and the first column. Call $I(G')$ the ideal 
generated by the  submaximal pfaffians of the  $G'$. 
As observed in Lemma  \ref{WL2},  the multiplication map by 
a general linear form $L$  determines a surjection  
$(R'/I(G'))_{d-1}\to (R'/I(G'))_d$ for all $d\geq tr(A')/2 -1$.
In particular, we get that $LI(H)+I(G')$ coincides with $R'$
in all degrees $tr(B)/2+1$.

Let $G$ be the matrix obtained from $H$ by multiplying
the first row and column by $L$. We have $\partial G=A$ and
moreover the ideal $I(G)$ generated by the submaximal 
pfaffians of $G$ contains $LI(H)+I(G')$. The claim follows
from Remark \ref{reduc}.\hfil\qed
\bigskip

We will need in the next section a technical results on 
submaximal pfaffians of skew-symmetric matrices of odd size
(see Definition \ref{sub}).

\begin{prop}\label{step3}
{Let} $A=(a_{ij})$ be a non-negative ordered
symmetric homogeneous integer matrix of odd size 
$(2k-1)\times (2k-1)$, $k>1$.
For a general choice of $k$ matrices of ternary forms
$G_1,\dots, G_k$ with $\partial G_i=A$, the
submaximal pfaffians of all the $G_i$'s generate an ideal
$I$ which coincides with the ring $R'$ in degree $d\geq(a_{11}+tr(A))/2$.
\end{prop}
\begin{proof} It is enough to prove the result for $d=
(a_{11}+tr(A))/2$.

Assume that all the entries of $A$ are equal to $a$. Then 
start with a general $2k\times 2k$ skew-symmetric matrix $G$, 
with all entries of degree $a$, and consider the matrices $G_i$ 
obtained from $G$ by erasing the $i$-th row and column. Then
any submaximal pfaffian of $G$ is a submaximal pfaffian of some $G_i$, 
and viceversa. Thus the ideal $I$, generated by all the submaximal
pfaffians of the matrices $G_1,\dots, G_k$, coincides  with the ideal generated 
by the submaximal pfaffians of $G$. By Remark \ref{reduc} and
 Theorem \ref{tre}, it follows that $I$ coincides with the whole ring
 in degree $(a+tr(A))/2$, which is the degree of the pfaffian of $G$.
 These matrices $G_1,\dots, G_k$ are quite special, but the claim follows
 for a general choice of the $G_i$'s, by semicontinuity.
 
 In the general case, assume that $A$ is ordered and 
 let $B=(b_{ij})$ be the matrix obtained from $A$ by decreasing 
 the first row and column by $1$. We may assume
 that the claim holds from $B$.
Notice that $(b_{11}+tr(B))/2= d-2$.
 Take $k$ general skew-symmetric matrices of forms
 $H_1,\dots, H_k$, with $\partial H_i=B$. Then the ideal $I'$
 generated by the submaximal minors of the $H_i$'s coincides with
 $R'$ in degree  $(b_{11}+tr(B))/2=(a_{11}+tr(A))/2 -2$. 
 Let $I_1$ be the ideal generated by
 the submaximal pfaffians of $H_1$. By Lemma \ref{WL2}
 the multiplication map $(R'/I_1)_{d-2}\to (R'/I_1)_{d-1}$
 surjects. Let $G_2,\dots, G_k$ be the matrices obtained from
 the $H_i$'s by multiplying the first  row and column by 
 a general linear form $L$. Then $\partial G_i=A$
 and the ideal $I''$ generated by the submaximal pfaffians
 of $H_1,G_2,\dots, G_k$ contains $LI'+I_1$, thus it
 coincides with $R'$ in degree $d-1$. Let now $I_2$ be 
 the ideal generated by the submaximal
 pfaffians of $G_2$. By Lemma \ref{WL2}
 the multiplication map $(R'/I_2)_{d-1}\to (R'/I_2)_d$
 surjects. Let $G_1$ be the matrix obtained from
$H_1$ by multiplying the first  row and column by 
 a general linear form $L$. Then $\partial G_1=A$
 and the ideal generated by the submaximal minors of 
 $G_1,\dots, G_k$ contains $I_2+LI''$, thus it
 coincides with $R$ in degree $d$. Hence the claim holds
 for $A$.
 
 The proof is concluded by observing that any 
 symmetric, homogeneous matrix $A$
 reduces to a matrix with constant entries by steps
 consisting in subtracting $1$ from one row and one column.
\end{proof}

\section{The four by four case}

We move now to the case of quaternary forms 
(and surfaces in $\PP^3$). For $4\times 4$ matrices
a complete answer to the problem of the pfaffian representation of forms
is given by the following.

\begin{thm}\label{4x4}
Let $A=(a_{ij})$ be a $4\times 4$ symmetric 
homogeneous matrix of non-negative
integers. Let $d=tr(A)/2$. Then a general form of 
degree $d$ in $\CC[x_1,\dots,x_4]$
is the sum of two pfaffians of skew-symmetric matrices, 
whose degree matrix is $A$.
\end{thm} 
\begin{proof} Let $M=(M_{ij})$ be a general $4\times 4$ 
skew-symmetric matrix of forms, 
with $\partial M=A$. If $F$ is the pfaffian of 
$M$, then Proposition \ref{tg} tells us
that the tangent space at $F$ of the variety $V$ 
(of forms which are pfaffians of matrices 
with degree matrix $A$) is generated by the  
$2\times 2$ submaximal pfaffians of $M$.
These pfaffians correspond to the six entries $M_{12},M_{13},M_{14},M_{23},M_{24},M_{34}$,
thus they are six general forms, of degrees (respectively) $a_{12},a_{13},a_{14},a_{23},a_{24},a_{34}$,
The homogeneity of $A$ implies that 
$$a_{12}+a_{34} = a_{13}+a_{24} = a_{14}+a_{23} 
$$
Thus, after Remark \ref{reduc}, the claim follows if we prove that 
$12$ general forms, of degrees respectively
$$a_{12},a_{13},a_{14},a_{23},a_{24},a_{34},\ a_{12},a_{13},a_{14},a_{23},a_{24},a_{34}$$
generate the polynomial ring $\CC[x_1,\dots,x_4]$ in degree 
$a_{12}+a_{34}=tr(A)/2$.

On the other hand, it is a consequence of Theorem 2.9
of \cite{CG1} that already $8$ general forms of degrees respectively 
$a,b,c,e,a,b,c,e$, with $a+e=b+c$, generate 
$\CC[x_1,\dots,x_4]$ in degree $a+e$.
The claim thus follows by taking $a=a_{12}, 
b= a_{13}, c=a_{24},e=a_{34}$.
\end{proof}
 
After Beauville's work (see e.g. Theorem 2.1 of \cite{CF2}),
a form $F$ is the pfaffian of a matrix with degree matrix $A$
if and only if  the surface $F$ contains a complete intersection 
set of points, of type $a_{12},a_{13},a_{14}$.
 
Thus we just proved that:
 
\begin{cor} For any choice of numbers $d,a,b,c$ with $d>a,b,c$, 
a general form of degree $d$ is the sum of two forms $F_1$, $F_2$ 
corresponding to two surfaces, both containing a complete 
intersection set of points of type $a,b,c$.  
\end{cor}

Compare this statement with the results of \cite{CCG2}.
 
\begin{rem} We derived our statement from Theorem 2.9
 of \cite{CG1}, which geometrically proves that for any 
 $d,a,b$ with $a,b<d$ 
a general form of degree $d$ is the sum of two forms 
$F_1$, $F_2$ corresponding to two
surfaces, both containing a complete intersection curve of type $a,b$.

From this point of view, a geometric reading of the proof of 
Theorem \ref{4x4} seems straightforward.
 \end{rem}

\section{General quaternary forms as sum of pfaffians}
 
In this section, we want to extend the results for quaternary forms
and general degree matrices.

Let we denote by $R=\CC[x,y,z,t]$ the polynomial ring
in four variables and with $R'$ the quotient of $R$ by a general linear form 
(i.e. $R'$ is isomorphic to a polynomial ring in three variables).

We will need two results, derived directly from the previous sections.

\begin{lem}\label{WL4}
 Let $G$ be a general skew-symmetric 
$2k\times 2k$ matrix of linear quaternary 
forms. Call  $I$ the ideal generated by 
the submaximal $(2k-2)\times(2k-2)$ pfaffians of $G$.

Then the multiplication by a general linear form $L$ determines
a surjection $(R/I)_{d-1}\to (R/I)_d$
for all $d\geq k$. 
\end{lem}
\begin{proof} 
 Let $I$ be the ideal generated by the submaximal pfaffians 
 of $G$ and let $L$ be a general linear form. 
 Consider  the exact sequence
$$ (R/I)_{d-1}\stackrel L{\longrightarrow} (R/I)_d\to 
(R/(I,L))_d\to 0.$$ 
By Theorem \ref{tre} and Remark \ref{reduc}, the module on 
the right side is $0$, when $d= k$, and then also
 for all $d\geq k$.
The claim follows.
\end{proof}

Just with the same procedure, but using Proposition \ref{step3}
instead of Theorem \ref{tre}, we get the following.

\begin{lem}\label{step4}
Let Let $A=(a_{ij})$ be a non-negative ordered
symmetric homogeneous integer matrix of odd size 
$(2k-1)\times (2k-1)$, $k>1$.
For a general choice of $k$ matrices of quaternary forms
$G_1,\dots, G_k$ with $\partial G_i=A$, call $I$ the ideal 
generated by the submaximal pfaffians of all the $G_i$'s.
Then the multiplication by a general linear form $L$ determines
a surjection $(R/I)_{d-1}\to (R/I)_d$ for all $d\geq(a_{11}+tr(A))/2$.
\end{lem}

We consider first  the case of matrices of linear forms.

\begin{thm}\label{linmain} 
For a general choice of $k$ skew-symmetric matrices of linear forms
$H_1,\dots, H_k$ of size $2k\times 2k$, the submaximal
pfaffians of the $H_i$'s generate an ideal which coincides with
the polynomial ring $R$ in all degrees $d\geq k$. 
\end{thm}
\begin{proof}
Use induction on $k$. The case $k=2$ holds trivially
since the submaximal pfaffians correspond to the choice
of six general linear forms.

In the general case, by induction, all forms  of degree $k-1$ in four
variables sit in the ideal $I'$ generated by the submaximal pfaffians
of $k-1$ general skew-symmetric matrices of linear forms 
$G_1,\dots, G_{k-1}$,  of size $(2k-2)\times (2k-2)$.
Choose one general $2k\times 2k$ skew-symmetric matrix of linear 
forms $M$. By Lemma \ref{WL4}, if $I$ is the ideal 
 generated by the submaximal pfaffians
of $M$, then the multiplication by a general
linear form $L$ determines a surjection $(R/I)_{d-1}\to (R/I)_d$
for all $d\geq k$.

Let $H_1,\dots, H_{k-1}$ be the matrices obtained by the $G_i$'s
by adding, as first two rows and two columns, the vectors 
$(0\ L\ 0\ \dots \ 0)$ and $(-L\ 0\ 0\ \dots \ 0)$.
 Then the non-zero submaximal pfaffians 
of the $H_i$'s are the submaximal pfaffians of the
$G_i$'s multiplied by $L$. Thus the submaximal pfaffians
of the matrices $H_1,\dots, H_{k-1}, M$ generate an ideal  
which contains $I+LI'$, hence it coincides
with $R$ in degree $d\geq k$. The claim follows.
\end{proof}

\begin{thm}\label{submain} 
Fix a $2k\times 2k$
symmetric homogeneous matrix $A$ of non-negative integers.
Then for a general choice of $k$ matrices of  forms
$G_1,\dots, G_k$ with $\partial G_i=A$ for all $i$, the submaximal
pfaffians of the $G_i$'s generate an ideal which coincides with
the polynomial ring $R$ in degree $d\geq tr(A)/2$. 
\end{thm}
\begin{proof} We may assume $k\geq 2$, the case $k=1$ being trivial.
We make induction on the trace of $A$.

After Remark \ref{order}, we know that the entries in the
diagonal of the matrix $A$ are either all even or all odd.
In the first case, we use as basis for the induction the
null matrix $A$ (for which the claim is obvious). In the
latter case we use a matrix with all entries equal to $1$
(for which the claim follows by Theorem \ref{linmain}).

In the inductive step, let $A$ be ordered and call $B$
the matrix obtained by $A$ by subtracting $1$ from the first row and 
the first column (thus subtracting $2$ from the upper-left
element, so that $tr(B)=tr(A)-2$).
As the theorem holds for $B$, for a general choice
of skew-symmetric matrices $G_1,\dots, G_k$ with
$\partial G_i=B$, the ideal $I$ generated by the submaximal
pfaffians of the $G_i$'s coincides
with the ring $R$ in degree $d\geq tr(A)/2-1$.

Let $G'_i$ be the matrix obtained from $G_i$ by erasing
the first row and column and call $I_0$ the ideal
generated by the $(2k-2)\times(2k-2)$ pfaffians of all
the $G'_i$'s.  The degree matrix $A'=(a'_{ij})$
of the $G'_i$'s is the matrix obtained from $A$ by
deleting the first row and column. Since $A$ is ordered, 
we have $tr(A)\geq a'_{11}+tr(A')$. Thus, by Lemma \ref{step4}
the multiplication by a general linear form $L$
determines a surjection $(R/I_0)_{d-1}\to (R/I_0)_d$,
for  $d\geq tr(A')+a'_{11}$, hence also for $d\geq  tr(A)$. 

Let $H_i$ be the matrix obtained from $G_i$ by multiplying
the first row and column by a general linear form $L$.
Then $\partial H_i=A$ and the ideal $I'$ generated by
the submaximal pfaffians of the $H_i$'s contains
both $I_0$ and $LI$.

The claim follows.
\end{proof}

Our main result follows now from Remark \ref{reduc}.

\begin{thm}\label{maingen} 
Fix a $2k\times 2k$
symmetric homogeneous matrix $A$ of non-negative
integers, with trace $2d$.
Then a general form $F$ of degree $d$ in four
variables is the sum of the pfaffians of $k$ skew-symmetric
matrices of forms, with degree matrix $A$.
\end{thm}

In other words, we obtain $s(A)\leq k$.

\section{Sharpness}\label{impro}

It is very reasonable to ask how far is the bound 
for $s(A)$ given in Theorem \ref{maingen} to be sharp. 

This can be answered by computing the dimension
of the (projective) variety $V$ of forms which are the pfaffian
of a single skew-symmetric matrix $G$.

\begin{rem}
When $A$ is $4\times 4$, then the bound $s(A)=2$  is
sharp for most values of the entries of $A$, as explained in 
\cite{CCG2}.
\end{rem}

As the size $2k$ of $A$ grows, however, the given bound is
probably no longer sharp.

For instance, when all the entries of $A$ are $1$'s
(so we deal with skew-symmetric matrices of linear forms),
then formula 3.6 and the exact sequence 3.1 of \cite{CF2} 
show  that $\dim V = 4k^2 + o(k)$. So one  expects,
 at least for $k\gg 0$, that the $s$-secant variety
of $V$ fills the space of all forms of degree $k$ 
as soon as $s\geq k/24 + o(k)$. In other words,
we can state the following.

\begin{conj} A general form of degree
$k\gg 0$ can be expressed as a sum of $s$ pfaffians
of skew-symmetric $2k\times 2k$ matrices of linear forms,
for $$s = \frac k{24} +o(k).$$
\end{conj}

Notice that our bound $s=k$ is already linear in $k$,
but with a larger coefficient.
\medskip

The same phenomenon is expected to occur for other types of 
homogeneous symmetric matrices $A$ of large size.

For instance, if all the entries of $A$ are equal to
a constant $b\gg k$, 
then formula 3.6 and the exact sequence 3.1 of \cite{CF2}
tell us that $\dim(V) = k^2b^3/3 +o(b)$. Thus the expected
value $s(A)$ such that the $s(A)$-secant variety of $V$ fills 
the space of forms of degree $kb$ is
$s(A) = k/2 + o(k)$, which is (asymptotically) 
$1/2$ of our bound.
\smallskip

We hope that a refinement of our method will provide,
in a future, advances towards sharper bounds for $s(A)$.


\begin{thebibliography}{[MKRR0]}

\bibitem[AR96]{Adler}
A.~Adler and S.~Ramanan.
\newblock{\em Moduli of abelian varieties.}
\newblock  Lecture Notes in Mathematics, vol. 1644, 
Springer-Verlag, Berlin, 1996.

\bibitem[Bea00]{Beau}
A.~Beauville. 
\newblock{\em Determinantal hypersurfaces.}
\newblock{ Michigan Math. J.} 48 (2000), 39--64, 
\newblock Dedicated to William Fulton on 
the occasion of his 60th birthday.

\bibitem[BE77]{BE}
D.~A.~Buchsbaum and D.~Eisenbud. 
\newblock{\em Algebra structures for   finite free resolutions, and 
some structure theorems for ideals of  codimension {$3$}.}
\newblock Amer. J. Math. 99 (1977),  447--485.

\bibitem[CCG08]{CCG2}
E.~Carlini, L.~Chiantini and A.~V.~Geramita.
\newblock {\em Complete intersections on general hypersurfaces.}
\newblock {Michigan Math. J.} 57 (2008), 121--136.
  
\bibitem[CF09]{CF2}
L.~Chiantini and D.~Faenzi. 
\newblock{\em On general surfaces defined by an almost linear pfaffian.}
\newblock Geom. Ded. 142 (2009), 91-107. 

\bibitem[CG14]{CG1}
L.~Chiantini and A.~V.~Geramita.
\newblock {\em On the Determinantal Representation of Quaternary Forms.}
\newblock Commun. Alg. 42 (2014), 4948-4954.

\bibitem[CG15]{CG2}
L.~Chiantini and A.~V.~Geramita.
\newblock{\em Expressing a general form as a sum of determinants.}
\newblock Collect. Math. to appear. DOI: 10.10007/s13348-014-0117-8.

\bibitem[CM12]{CM} 
L.~Chiantini and J.~Migliore
\newblock {\em Determinantal representation  and subschemes 
of general plane curves.}
\newblock {Lin. Alg. Applic.} 436 (2012), 1001--1013.

\bibitem[F08]{Faen}
D.~Faenzi.
\newblock{\em A Remark on Pfaffian Surfaces and aCM Bundles.}
\newblock{Vector bundles and low codimensional subvarieties: 
state of the art and recent developments.}  
\newblock{Quaderni di Matematica} della Seconda Universit\'a 
di Napoli, 11 (2008).

\bibitem[Ha95]{H}
T.~Harima. 
\newblock{\em Characterization of Hilbert functions of Gorenstein Artin 
algebras with the Weak Stanley property.}
\newblock{Proc. AMS} 123 (1995), 3631--3638.

\bibitem[HL12]{HL12} D.~Henrion and J.~B.~Lasserre. 
\newblock{\em Inner approximations for polynomial matrix 
inequalities and robust stability regions.}
\newblock{IEEE Trans. Autom. Control.} 57 (2012), 1456--1467.

\bibitem[K78]{Kleppe} H.~Kleppe. 
\newblock{\em Deformation of Schemes Defined by Pfaffians.}
\newblock{J. Alg.} 43 (1978) 84--92.

\bibitem[KL80]{KL} H.~Kleppe and D.~Laksov. 
\newblock{\em The Algebraic Stucture and Deformations of Pfaffian 
Schemes.} \newblock{J. Alg.} 64 (1980), 167--189.

\bibitem[MKRR06]{MKRR} 
N.~Mohan~Kumar, P.~Rao and G.~Ravindra. 
\newblock{\em Four by Four Pfaffians.}
\newblock {Rend. Sem. Mat. Univ. Pol. Torino} 64 (2006), 471--477. 
 
\bibitem[V89]{V89}
V. Vinnikov.
\newblock {\em  Complete description of determinantal representations 
of smooth irreducible curves.}  
\newblock{Lin. Alg. Applic.} 125 (1989), 103-140.


\end{thebibliography}
\end{document}